\documentclass{svproc}
\usepackage{graphicx}
\usepackage{url}
\usepackage{bm}
\usepackage{amsmath,amssymb}
\usepackage{mathtools}
\usepackage{xcolor}
\usepackage{makecell}

\newcommand\scalemath[2]{\scalebox{#1}{\mbox{\ensuremath{\displaystyle #2}}}}
\newcommand{\half}{\frac{1}{2}}
\newcommand{\Div}{\mbox{\rm{\bf{div\,}}}}

\newcommand\changed[1]{{\color{black} #1}}

\begin{document}
\mainmatter              %
\title{Anisotropic Diffusion Stencils:\\
From Simple Derivations over Stability Estimates to ResNet Implementations
\thanks{This work has received funding from the European Research Council (ERC) 
under the European Union’s Horizon 2020 research and innovation programme 
(grant agreement no. 741215, ERC Advanced Grant INCOVID).}}
\titlerunning{Anisotropic Diffusion Stencils}  
\author{Karl Schrader \and Joachim Weickert \and Michael Krause}
\authorrunning{Karl Schrader, Joachim Weickert, and Michael Krause} 
\tocauthor{Karl Schrader, Joachim Weickert, Michael Krause}
\institute{Mathematical Image Analysis Group, 
 Faculty of Mathematics and Computer Science,\\
 Campus E1.7, Saarland University, 66041 Saarbrücken, Germany\\
\email{\{schrader,weickert,krause\}@mia.uni-saarland.de}
}

\maketitle              %

\begin{abstract}
Anisotropic diffusion processes with a diffusion tensor are  
important in image analysis, physics, and engineering. However, their 
numerical approximation has a strong impact on dissipative artefacts 
and deviations from rotation invariance. 
In this work, we study a large family of finite difference discretisations 
on a $3\times3$ stencil. We derive it by splitting 2-D anisotropic diffusion
into four 1-D diffusions. The resulting stencil class involves one free 
parameter and covers a wide range of existing discretisations. It comprises 
the full stencil family of Weickert et al.~(2013) and shows that their
two parameters contain redundancy.  
Furthermore, we establish a bound on the spectral norm of the matrix 
corresponding to the stencil. 
This gives time step size limits that guarantee stability of an explicit
scheme in the Euclidean norm.
Our directional splitting also allows a very natural translation of the
explicit scheme into ResNet blocks. Employing neural network libraries 
enables simple and highly efficient parallel implementations on GPUs. 
\end{abstract}
\section{Introduction}
Anisotropic diffusion models with a diffusion tensor have numerous
applications in physics and engineering. Moreover, they also play a 
fundamental role in image analysis \cite{We97}, where they are used 
for denoising, enhancement, scale-space analysis, and various 
interpolation tasks such as inpainting 
and superresolution.
Sophisticated nonlinear models with appropriate directional behaviour 
can close interrupted structures and maintain or create sharp edges.  
However, to achieve results with only few dissipative artefacts 
and good rotation invariance, appropriate numerical approximations
are needed. 
They should also come with provable stability guarantees and
lead in a natural way to efficient implementations. Ideally they
should also exploit the impressive parallelisation potential of 
modern GPUs. 
The goal of our contribution is to address these numerical issues. 

\medskip
\noindent
\textbf{Our Contributions.}
Motivated by image analysis applications, where one has a regular
pixel grid and aims at simple numerical algorithms, we consider
finite difference approximations on a $3 \times 3$ stencil. However, 
our results are also useful for anisotropic diffusion problems
in other areas. Our contributions are threefold:

First, we study space discretisations of a general anisotropic 
diffusion operator on a $3 \times 3$ stencil. 
They split the 2-D anisotropic process into four 1-D diffusions. This 
class has one free parameter that can be used for quality optimisation.
It covers the two-parameter stencil family of Weickert et al.~\cite{WWW13}, 
while removing its parameter redundancy and offering a simpler derivation. 
Moreover, it subsumes many previous discretisations with second-order 
consistency. 

Our second contribution consists of a detailed stability analysis, where 
we establish fairly tight bounds on the spectral norm of the matrix
associated with the stencil family. It allows to derive time step
size restrictions for the corresponding explicit scheme (and
accelerations that rely on it). 

Last but not least, our stencil derivation based on a directional 
splitting enables the translation of the explicit anisotropic diffusion
scheme into a ResNet block~\cite{HZRS16}, which is a highly popular 
component of neural networks. This showcases that ideas are often 
shared between numerical schemes and neural architectures. More
importantly, it allows simple and fast parallel implementations 
of anisotropic diffusion on GPUs using neural network libraries 
such as PyTorch. 

\medskip
\noindent
\textbf{Related Work.}
Many finite difference discretisations for anisotropic diffusion
processes exist in the literature. Often they use spatial discretisations 
on a $3 \times 3$ stencil with consistency order two. The stencil class
of Weickert et al.~\cite{WWW13} comprises seven of them. 
Our findings offer a simpler derivation and representation of this
family. Moreover, we extend the results from \cite{WWW13} by establishing 
concrete time step size limits for explicit schemes, connecting these 
algorithms to neural networks, and exploring simple and efficient 
parallelisations.

Our stencil family originates from a splitting {2-D} anisotropic diffusion
into four {1-D} diffusions along fixed directions. Earlier splittings of 
this type intended to derive discretisations that are stable in the 
maximum norm~\cite{We97,MN01a}. 
In general this is only possible for fairly mild anisotropies 
\cite{We97}. We consider stencils that offer stability in the 
Euclidean norm for all anisotropies.

Recent works \cite{ASAPW21,ASWPA22,RDF20,RH20} connect explicit schemes 
for partial differential equations (PDEs) to the ResNet~\cite{HZRS16} 
architecture. For example, Alt et al.~\cite{ASAPW21} show that evolutions of 
discretised 1-D diffusion models with a scalar-valued diffusivity can be 
represented as ResNet blocks. 
In \cite{ASWPA22}, they also explore the 2-D anisotropic case. 
However, their methodology is limited to evolution equations that arise as
gradient descent of an energy functional. This excludes popular methods like 
edge-enhancing diffusion~\cite{We97}, for which Welk~\cite{We21} has 
shown that no energy functional exists. We can translate these 
methods as well.

\medskip
\noindent
\textbf{Organisation of the Paper.} 
In Section~2, we derive a class of finite difference discretisations on a 
$3\times3$ stencil. We establish stability results for the corresponding
explicit scheme in Section~3. The fourth section shows how our splitting 
into 1-D diffusions leads to a translation of this scheme to a ResNet 
architecture, and it analyses its performance on a GPU. We conclude our 
paper in Section~5.

\section{Discretising Anisotropic Diffusion with the $\bm \delta$-Stencil }
In this section, we study a simple and fairly general approach for a space 
discretisation of anisotropic diffusion on a $3 \times 3$ stencil. It is based 
on a directional splitting into four 1-D diffusion processes, which we discuss
first. 

\medskip
\noindent
\textbf{1-D Diffusion.}
To denoise a 1-D signal \changed{$f: [a, b] \to \mathbb{R}$}, one can create
simplified versions $\{u(x, t)\,|\, t \ge 0\}$ of it with the
nonlinear diffusion process~\cite{PM90}
\begin{align}
\partial_t u 
&\;=\; \partial_x\Big(g\big((\partial_x u)^2\big)\, \partial_x u\Big) 
  & (t> 0), 
\label{eq:diffusion_1d_cont}\\
u(x,0) &\;=\; f(x)\;.
\end{align}
Larger diffusion times $t$ correspond to more pronounced simplifications.
The diffusivity $g: \mathbb{R}\to (0,1]$ is a function that decreases in 
its argument $(\partial_x u)^2$ in order to preserve discontinuities. 
At the domain boundaries $a$ and $b$, we impose reflecting boundary conditions.
To prepare for the later translation to a neural architecture, we introduce 
the flux function $\,\Phi(\partial_x u) = g((\partial_x u)^2)\,\partial_x u$. 
It leads to the evolution equation 
$\,\partial_t u = \partial_x(\Phi(\partial_x u))$.

A finite difference discretisation of this 1-D process serves as building 
block for discretising anisotropic diffusion. To obtain a discrete signal 
$\bm u = (u_i) \in \mathbb{R}^N$, we sample $u$ with grid size $h$. We 
discretise the derivatives with a forward difference in time and for 
the inner spatial derivative, and a backward difference for the outer one. 
This leads to the explicit scheme
\begin{equation}
\frac{u^{k+1}_i - u^k_i}{\tau} \;=\; 
\frac{1}{h}\left(
  \Phi\left(\frac{u^k_{i+1} - u^k_{i}}{h}\right) 
- \Phi\left(\frac{u^k_{i} - u^k_{i-1}}{h}\right) 
\right), 
\label{eq:iso_explicit}
\end{equation}
where $\tau>0$ is the time step size, $i$ denotes the location, and $k$ 
the time level.

\medskip
\noindent
\textbf{Anisotropic Diffusion.}
In image analysis, anisotropic diffusion with a diffusion tensor~\cite{We97} 
creates filtered versions $u(\bm x, t)$ of a scalar-valued (i.e.~greyscale) 
image $f(\bm x)$ by evolving it with the PDE 
\begin{align}
\partial_t u &\;=\; \Div(\bm D \, \bm\nabla u), \qquad 
\bm D \;=\; \begin{pmatrix}a \,&\, b \\ b \,&\, c \end{pmatrix} 
\label{eq:aniso_continuous}
\end{align}
where we initialise $u(\bm x, 0)$ with $f(\bm x)$ and use reflecting 
boundary conditions. The diffusion tensor $\bm D \in \mathbb{R}^{2 \times 2}$ 
is symmetric, positive semidefinite with at least one positive eigenvalue.
$\bm D$ may depend on Gaussian-smoothed first order derivatives of the 
evolving image $u$. This allows to enhance edges and coherent flow-like 
structures by smoothing along them, but not perpendicular to them~\cite{We97}.

\medskip
\noindent
\textbf{The $\bm\delta$-Stencil.}
The discrete setting considers images $\bm f, \bm{u}^k \in \mathbb{R}^N$
obtained by sampling $f$ and $u(.,k\tau)$ with a grid size of $h$ and
arranging the pixel values into column vectors. The key idea of our 
discretisation is the decomposition of an anisotropic 2-D diffusion 
process into a sum of four nonlinear 1-D diffusions along the axial 
and diagonal directions
\begin{equation}
\bm e_0 = \begin{pmatrix}1\\0\end{pmatrix}, \quad
\bm e_1 = \frac{1}{\sqrt{2}}\begin{pmatrix}1\\1\end{pmatrix}, \quad
\bm e_2 = \begin{pmatrix}0\\1\end{pmatrix}, \quad
\bm e_3 = \frac{1}{\sqrt{2}}\begin{pmatrix}-1\\1\end{pmatrix}.
\end{equation}
We determine directional diffusivities $w_0, \dots, w_3$ for the corresponding 
directions by solving the system of three equations with four unknowns arising 
from
\begin{equation}
\Div \left(\begin{pmatrix}a\,&\,b\\b\,&\,c\end{pmatrix}\bm\nabla u\right) 
\stackrel{!}{\;=\;} 
\sum_{i=0}^{3} \partial_{\bm e_i}\left(w_i\, \partial_{\bm e_i} u\right).
\end{equation}
Its solution has one free parameter which we call $\delta$:
\begin{equation}
w_0 = a-\delta, \qquad
w_1 = \delta+b, \qquad
w_2 = c-\delta, \qquad
w_3 = \delta-b \,.
\end{equation}
All four 1-D diffusion processes can be discretised as before in 
\eqref{eq:iso_explicit}. Each direction uses three pixels of its 
$3\times 3$ neighbourhood. Discretising e.g.
$\partial_{\bm e_1}(w_1\,\partial_{\bm e_1} u)$ in the pixels 
$(i\!-\!1, j\!-\!1)$, $(i,j)$, and $(i\!+\!1, j\!+\!1)$ at distance 
$h\sqrt{2}$ gives 
\begin{equation}
\frac{1}{h\sqrt{2}}\left((\delta+b)_{i+\frac{1}{2}, j+\frac{1}{2}} 
\frac{u_{i+1, j+1}-u_{i,j}}{h\sqrt{2}} - 
(\delta+b)_{i-\frac{1}{2},  j-\frac{1}{2}} 
\frac{u_{i, j}-u_{i-1,j-1}}{h\sqrt{2}}\right).
\end{equation}
Incorporating all four directions yields the following 
{\bf ${\bm \delta}$-stencil} for $\Div(\bm D \, \bm \nabla u)$:
\vspace{1mm}
\renewcommand{\arraystretch}{1.5}
\begin{equation}
\scalemath{.88}{\kern-.6em
\frac{1}{h^2} \cdot \,
\begin{array}{|c|c|c|} \hline
  \mbox{\rule[-4ex]{0pt}{8ex}}
  \begin{array}{l}
    \tfrac{1}{2}\,(\delta-b)_{i-\half,j+\half}
  \end{array} &
  \begin{array}{l}
    (c-\delta)_{i,j+\half}
  \end{array} &
  \begin{array}{l}
    \tfrac{1}{2}\,(\delta+b)_{i+\half,j+\half}
  \end{array}
  \\ \hline
  \mbox{\rule[-10ex]{0pt}{20ex}}
  \begin{array}{l}
    (a-\delta)_{i-\half,j}
  \end{array} ~ &
  \begin{array}{c}
    {}-(a-\delta)_{i+\half,j}\,-\,(a-\delta)_{i-\half,j}\\ 
    {}-\tfrac{1}{2}\,(\delta+b)_{i+\half,j+\half}
    \,-\,\tfrac{1}{2}\,(\delta+b)_{i-\half,j-\half}\\ 
    {}-(c-\delta)_{i,j+\half}\,-\,(c-\delta)_{i,j-\half}\\
    {}-\tfrac{1}{2}\,(\delta-b)_{i-\half,j+\half}
    \,-\,\tfrac{1}{2}\,(\delta-b)_{i+\half,j-\half}
  \end{array} ~ &
  \begin{array}{l}
    (a-\delta)_{i+\half,j}
  \end{array} ~
  \\ \hline
  \mbox{\rule[-4ex]{0pt}{8ex}}
  \begin{array}{l}
    \tfrac{1}{2}\,(\delta+b)_{i-\half,j-\half}
  \end{array} &
  \begin{array}{c}
    (c-\delta)_{i,j-\half}
  \end{array} &
  \begin{array}{l}
    \tfrac{1}{2}\,(\delta-b)_{i+\half,j-\half}
  \end{array}
  \\ \hline
\end{array}}\label{eq:delta_stencil}
\end{equation}%
\renewcommand{\arraystretch}{1}%

\vspace{1.7mm}
\noindent
where the $x$-axis points to the right, and the $y$-axis to the top.
We assume that the diffusion tensor $\bm D$ is available in the 
staggered grid locations $\left(i \pm \frac{1}{2}, j \pm \frac{1}{2}\right)$. 
This is fairly natural if it relies on first-order derivatives, which 
can be computed with central differences in a $2 \times 2$ 
neighbourhood~\cite{WWW13}.
We  obtain values in $\left(i \pm \frac{1}{2}, j\right)$ and
$\left(i, j \pm \frac{1}{2}\right)$ by averaging:
\begin{equation}
 (a-\delta)_{i \pm \frac{1}{2}, j} \;=\;
 \half \left((a-\delta)_{i \pm \frac{1}{2}, j + \frac{1}{2}}
 \;+\; (a-\delta)_{i \pm \frac{1}{2}, j - \frac{1}{2}}\right) 
 \label{eq:ave}
\end{equation}
and similar for $(c-\delta)_{i, j \pm \frac{1}{2}}$.
Then the $\delta$-stencil family has consistency order two.

\medskip
\noindent
{\bf Incorporation of the Stencil Family of Weickert et al.~\cite{WWW13}.}
With \eqref{eq:ave} and $\delta = \alpha a + \beta b + \alpha c$, 
the $\delta$-stencil family comprises that of Weickert et al.~\cite{WWW13} 
that uses two parameters $\alpha$ and $\beta$. This shows that the parameters
of the latter contain redundancy which we remove with the $\delta$-stencil. 
Moreover, our stencil derivation is simpler than the one in \cite{WWW13} 
that has been obtained by discrete energy minimisation. In \cite{WWW13} 
it is shown that these stencils comprise seven discretisations from the 
literature. Since we are not aware of any second-order accurate 
discretisations on a $3 \times 3$ stencil that is not covered by this 
class, the $\delta$-stencil family may even be more general.
 
\section{Stability Theory for the $\bm\delta$-Stencil}
Let the \changed{symmetric} matrix $\bm A = \bm A(\bm u^k)$ act on an image 
$\bm u^k$ locally by applying the space-variant $\delta$-stencil. 
Weickert et al.~\cite{WWW13} have already established that $\bm A$ is 
negative semidefinite for $\alpha \le \half$ and $|\beta| \leq 
1-2\alpha$. They have also replaced $\beta$ by a parameter
$\gamma$ such that $\beta = \gamma (1\!-\!2\alpha)\,\text{sgn}(b)$
and $|\gamma|\leq 1$. Choosing $\alpha$ close to $\half$ and $\gamma$ close 
to $1$ improves rotation invariance and reduces dissipativity in experiments
\cite{WWW13}. In practice, parameters $\alpha < 0$ are irrelevant and
make a stability analysis more complicated. Thus, we exclude them from
now on. 

Consider an explicit anisotropic diffusion scheme  
$\,\bm u^{k+1} = (\bm I +\tau \bm A(\bm u^k))\,\bm u^k\,$ with unit 
matrix $\bm I$, time step size $\tau>0$, and a negative semidefinite 
matrix 
\changed{$\bm A \neq\bm0$ such that for the spectral norm $\rho(\bm 
A)>0$ holds}. 
Then stability in the Euclidean norm in terms of
$\,\|\bm u^{k+1}\|_2 \leq \|\bm u^k\|_2\,$ holds if 
\begin{equation}
\tau \;\leq\; \frac{2}{\rho(\bm A)} \,. \label{eq:tau_limit}
\end{equation}
We can bound $\rho(\bm A)$ as follows:

\begin{theorem}[Bound on Spectral Norm] \label{th:a_norm}
Let the eigenvalues of $\bm D$ be given by $\lambda_1 \ge \lambda_2 \ge 0$.
Assume that $\,\delta = \alpha(a\!+\!c) + \beta b\,$ where 
$\,\beta=\gamma(1\!-\!2\alpha)\,\textup{sgn}(b)\,$ for 
$\alpha \in [0, \half]$ and $|\gamma|\leq 1$.
Then the spectral norm of the matrix $\bm A$ satisfies
\begin{equation}
\rho(\bm A) \;\leq\; 
\frac{4\,(1\!-\!\alpha)\,(\lambda_1 \!+\! \lambda_2) 
 \,+\, 2\,(1-\gamma\,(1\!-\!2\alpha))\,(\lambda_1\!-\!\lambda_2)}{h^2}\;.
\end{equation}
\end{theorem}

\begin{proof}
For the considered choice of $\alpha$ and $\beta$, we know from~\cite{WWW13} 
that the symmetric matrix $\bm A$ is negative semidefinite. Thus, its 
spectral norm is determined by its smallest eigenvalue $\lambda_{\min}$ 
as $\,\rho(\bm A) = -\lambda_{\min}(\bm A)\,$.
Let us now bound $\lambda_{\min}(\bm A)$ with Gershgorin's circle 
theorem~\cite{We18}. 
As $\bm A$ applies the $\delta$-stencil, this theorem states that the smallest 
eigenvalue of $\bm A$ is bounded from below by the central stencil entry, 
minus the sum of absolute values of all other entries. Using \eqref{eq:ave} 
and grouping all terms by the four diffusion tensor locations 
$(i \pm \half, j \pm \half)$ gives
\begin{equation}
\scalemath{0.86}{
\begin{matrix*}[l]
  \rho(\bm A) \;\leq\; \frac{1}{2h^2}\max\limits_{\smash{a,b,c}} &\big(
  &((a-\delta) + |a-\delta| + (\delta+b) + |\delta+b| 
  + (c-\delta) + |c-\delta|)_{i-\half,j-\half} \vspace{-.5mm} \\
  &+&((a-\delta) + |a-\delta| + (\delta-b) + |\delta-b| 
  + (c-\delta) +  |c-\delta|)_{i-\half,j+\half}  \\
  &+&((a-\delta) + |a-\delta| + (\delta-b) + |\delta-b| 
  + (c-\delta) + |c-\delta|)_{i+\half,j-\half}  \\
  &+&((a-\delta) + |a-\delta| + (\delta+b) + |\delta+b| 
  + (c-\delta) + |c-\delta|)_{i+\half,j+\half} \, \big).
\end{matrix*}\label{eq:rho_limit_1}}
\end{equation}
Next, we bound the right hand side from above by assuming that the diffusion 
tensors at the different locations are independent. The notation 
\begin{equation}
M_\pm \;\coloneqq\;  
 (a-\delta) + |a-\delta| + (\delta\pm b) + |\delta\pm b| + (c-\delta) + 
 |c-\delta|
\end{equation} 
allows us to rewrite the bound as
\begin{equation}
\rho(\bm A) \;\leq\; 
 \frac{1}{h^2}\big(\max\limits_{a,b,c}(M_+) + \max\limits_{a,b,c} (M_-)\big)\,. 
\label{eq:rho_limit_2}
\end{equation}
In the following, we determine the maximum of $M_+$. Calculations for 
$M_-$ are analogous. 
\pagebreak[1] 
Notice that in $M_+$, the three terms 
$\,a-\delta$, $\,\delta + b$, and $\,c - \delta$ appear 
pairwise with their 
absolute values. This will simplify the calculation of the maximum. 
Consider the sum of the three terms:
\begin{equation}
m_+ \;\coloneqq\; (a-\delta) \,+\, (\delta+b) \,+\, (c-\delta) 
    \;=\; a+b+c-\delta \,. \label{eq:mplus}
\end{equation}
If $m_+$ has a maximum in which $a - \delta$, $\delta + b$, and $c - 
\delta$ are all nonnegative, then $\max(M_+) = 2\,\max(m_+)$, since 
$\,x+|x|=2x$ for $x \ge 0$.
We now proceed to show that such a maximum of $m_+$ exists. To this end, we 
rewrite the entries $a$, $b$, and $c$ of the positive semidefinite diffusion 
tensor~$\bm D$ in terms of its normalised eigenvectors $(u, v)^\top$, 
$(v, -u)^\top$ and their eigenvalues $\lambda_1 \geq \lambda_2 \geq 0$:
\begin{align}
 a = \lambda_1 u^2 + \lambda_2 v^2\,, \qquad
 b = (\lambda_1\!-\!\lambda_2)\, uv\,, \qquad
 c = \lambda_2 u^2 + \lambda_1 v^2\,. \label{eq:abc_decomp}
\end{align}
The possible ranges for eigenvalues may differ between diffusion models. 
Therefore, we determine the maximum of $m_+$, and by extension our limit on 
$\rho(\bm A)$, as a function of $\lambda_1$ and $\lambda_2$. This leaves the 
entries $u, v$ of the eigenvectors as the only variables to maximise $m_+$ 
over. Using \eqref{eq:abc_decomp} in \eqref{eq:mplus} gives
\begin{align}
\max_{u,v}\, (m_+) 
&\;=\; \max_{u,v}\big((1\!-\!\alpha)\,(\lambda_1 \!+\! \lambda_2) \,+\,   
  (1\!-\!\beta)(\lambda_1\!-\!\lambda_2)\,uv\big) \nonumber\\
&\hspace{-14mm} \;=\; \max_{u,v}\big((1\!-\!\alpha)\,
  (\lambda_1 \!+\! \lambda_2) 
  \,+\, (1\!-\!\gamma(1\!-\!2\alpha)\,\text{sgn}(uv))\,
  (\lambda_1\!-\!\lambda_2)\,uv \big) \nonumber\\
&\hspace{-14mm} \;=\; \max_{u,v}\begin{cases}
  (1\!-\!\alpha)\,(\lambda_1 \!+\! \lambda_2) + 
  \underbrace{(1-\gamma(1\!-\!2\alpha))}_{\geq 0}
  \underbrace{(\lambda_1\!-\!\lambda_2)}_{\geq 0}
  \underbrace{uv}_{> 0} & \mbox{for }\,uv > 0,\\
  (1\!-\!\alpha)\,(\lambda_1 \!+\! \lambda_2) + 
  \underbrace{(1+\gamma(1\!-\!2\alpha))}_{\geq 0}
  \underbrace{(\lambda_1\!-\!\lambda_2)}_{\geq 0}
  \underbrace{uv}_{\leq 0} & \mbox{for }\,uv\leq 0.\\
\end{cases}
\end{align}
The case where $uv> 0$ gives always larger results than the second one. 
Thus, 
\begin{equation}
\max_{u,v}\, (m_+) 
 \;=\; \max_{u,v} \big((1\!-\!\alpha)\,(\lambda_1 \!+\! \lambda_2) \,+\, 
  (1-\gamma(1\!-\!2\alpha))\,(\lambda_1\!-\!\lambda_2)\,uv\big)\,.
\end{equation}
We maximise the second term by maximising $uv$. For our 
normalised eigenvectors, $u^2+v^2=1$ holds. Hence, $\max(uv)=\half$ 
for $u=v=\pm \frac{1}{\sqrt{2}}$. 
Since we only need the maximal function value, we can consider only
$u=v=\frac{1}{\sqrt{2}}$. This gives
\begin{equation}
\max_{u,v}\, (m_+) 
 \;=\; (1\!-\!\alpha)\,(\lambda_1 \!+\! \lambda_2) 
  \,+\, \frac{1}{2}\,(1-\gamma(1\!-\!2\alpha))\,(\lambda_1\!-\!\lambda_2)\,.    
  \label{eq:mplusMax}
\end{equation}
We are not able to draw conclusions about the maximum of $M_+$ from the maximum 
of $m_+$ yet. It remains to show that $a - \delta$, $\delta + b$, and 
$c - \delta$ are all nonnegative in our maximum with 
$u=v=\frac{1}{\sqrt{2}}$. We start with $a-\delta$ and use 
\eqref{eq:abc_decomp}:
\begin{align}
 (a-\delta)|_{u=v=\frac{1}{\sqrt{2}}}
 &\;=\; 
  \big(a - \beta b - \alpha(a+c)\big)|_{u=v=\frac{1}{\sqrt{2}}} \nonumber\\
 &\;=\; \frac{1}{2} \, \big((\lambda_1 + \lambda_2) 
 -\gamma\underbrace{(1-2\alpha)}_{\geq0}\underbrace{(\lambda_1-\lambda_2)}_{\geq
   0} \,-\, 2\alpha(\lambda_1 + \lambda_2) \big)\nonumber\\[-0.5mm]
 &\;\geq\; \frac{1}{2} \, \big((\lambda_1 + \lambda_2) 
 -(1-2\alpha)(\lambda_1-\lambda_2) -2\alpha(\lambda_1 + \lambda_2) \big)
 \nonumber\\[1mm]
 &\;=\; (1-2\alpha)\, \lambda_2 \;\geq\; 0\,.
\end{align}
In a similar way, one shows $(c-\delta)|_{u=v=\frac{1}{\sqrt{2}}} \geq 0$.
For $b + \delta$ we verify  
\begin{align}
  (b+\delta)|_{u=v=\frac{1}{\sqrt{2}}} 
  &\;=\; \big((1+\beta)\,b + \alpha(a+c)\big)|_{u=v=\frac{1}{\sqrt{2}}}
    \nonumber\\
  &\;=\; \frac{1}{2}\,\underbrace{(1+\gamma(1-2\alpha))}_{\geq 
    0}\underbrace{(\lambda_1-\lambda_2)}_{\geq 0} \,+\;
  \alpha \, (\lambda_1+\lambda_2) \;\geq\; 0\,.
\end{align}
As all three terms are nonnegative in the maximum, we can conclude that
\begin{equation}
\max_{u,v}(M_+) \,=\, 2\max_{u,v}(m_+) 
 \,=\, 2\,(1\!-\!\alpha)\,(\lambda_1\!+\!\lambda_2)
 + (1-\gamma(1\!-\!2\alpha))\,(\lambda_1\!-\!\lambda_2)\,.
\end{equation}
Analogous computations lead to the same maximum for $M_-$. Inserting both into 
\eqref{eq:rho_limit_2} produces the claimed bound on the spectral norm. 
\qed 
\end{proof}

\smallskip
\noindent
Using Theorem \ref{th:a_norm} within \eqref{eq:tau_limit} directly 
gives the following time step size limit:
\begin{corollary}[Stability of Explicit Scheme] \label{th:ex}
An explicit anisotropic diffusion scheme 
$\,\bm u^{k+1} = \left(\bm I +\tau \bm A(\bm u^k)\right)\bm u^k$,
where $\bm A$ satisfies the assumptions of Theorem \ref{th:a_norm}
with $\lambda_1>0$, is stable in the Euclidean norm for 
\begin{equation}
 \tau \;\le\; \frac{h^2}{2\,(1\!-\!\alpha)\,(\lambda_1\!+\!\lambda_2) \,+\,
     (1-\gamma\,(1\!-\!2\alpha))\,(\lambda_1\!-\!\lambda_2)}\,.
\label{eq:bound}
\end{equation}
\end{corollary}

\noindent
While our proof does not guarantee that this bound is strict, our 
practical experience suggests that it is. 
Corollary \ref{th:ex} covers two important special cases:
\begin{enumerate}
\item
In the {\bf homogeneous diffusion case} with $\lambda_1=\lambda_2=1$, this
time step size limit simplifies to $\,\tau \le \frac{h^2}{4(1-\alpha)}$.
Moreover, setting $\alpha \coloneqq 0$ turns the $\delta$-stencil into
the standard five point approximation of the Laplacian, which leads
to the well-known 2-D time step size limit $\,\tau \le \frac{h^2}{4}$; see
e.g.~\cite{MM05}.

\vspace{1.6mm}
\item
In the {\bf maximally anisotropic case} with $\lambda_1=1$, $\lambda_2=0$, 
and $\gamma=1$, one performs 1-D diffusion along one eigendirection of 
$\bm D$. Then \eqref{eq:bound} becomes $\,\tau \le \frac{h^2}{2}$. 
In spite of being in a 2-D setting, this coincides with the typical 1-D 
time step size limit \cite{MM05}, which is less restrictive. This shows 
that our scheme 
takes full advantage of the anisotropy. In image analysis, this result is 
relevant for coherence-enhancing nonlinear diffusion filters \cite{We97}.
Similar findings have also been made with a recent numerical scheme
for a maximally anisotropic backward parabolic PDE \cite{SW22}.
\end{enumerate}

\section{Translating Anisotropic Diffusion into ResNets}

Let us now interpret our explicit scheme in the context of neural networks.
This extends the result of Alt et al.~\cite{ASAPW21} from the 1-D 
setting to the 2-D anisotropic case. We start by presenting their translation 
of 1-D diffusion into a ResNet block, and then build the anisotropic 
ResNet block from there. 

\medskip
\noindent
\textbf{ResNets.} Residual networks (ResNets)~\cite{HZRS16} are very popular 
neural network architectures. They use ResNet blocks that compute an
output $\bm u^{k+1}$ from an input $\bm u^k$ by
\begin{equation}
\bm u^{k+1} \;=\; 
\sigma_2\left(\bm u^k + \bm K_2\, \sigma_1
  \left(\bm K_1\bm u^k + \bm b_1\right)+\bm b_2\right) 
\label{eq:resnet_block}
\end{equation}
for discrete convolution kernels $\bm K_1, \bm K_2$, bias vectors 
$\bm b_1, \bm b_2$, and nonlinear activation functions $\sigma_1, \sigma_2$
such as the ReLU function $\sigma(x)=\max\{x,0\}$. 
Adding the input $\bm u^k$ before applying the second activation $\sigma_2$ 
helps to avoid vanishing gradients and to improve stability. This allows to 
train very deep networks.
\begin{figure}[!tb]
\centering
\begin{tabular}{ccc}
\includegraphics[height=60mm]{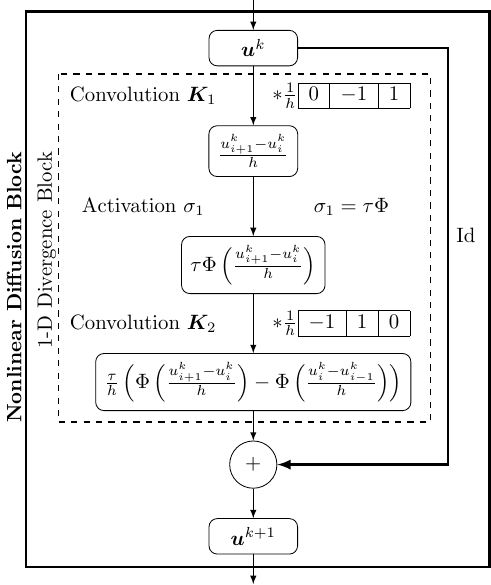}
&\hspace*{5mm}&
\includegraphics[height=60mm]{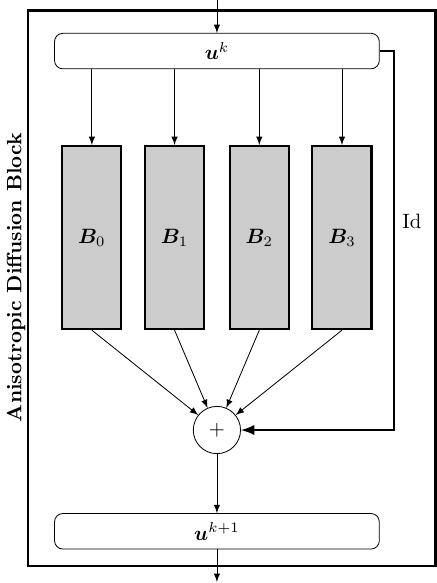}
\end{tabular}
\caption{\textbf{(a) Left}: Translation of 1-D nonlinear diffusion into a 
  ResNet block. Adapted from~\cite{ASAPW21}. 
  \textbf{(b) Right}: Anisotropic diffusion as a ResNet block with a sum 
  of four 1-D divergence blocks. The blocks $\bm{B}_0,...,\bm{B}_3$ correspond 
  to the directions $\bm{e}_0,...,\bm{e}_3$.}
\label{fig:isotropicblock}
\end{figure}

\medskip
\noindent
\textbf{Translating 1-D Diffusion into ResNets.} The basic translation 
of 1-D diffusion into ResNets is surprisingly simple \cite{ASAPW21,RH20}: 
In vector notation, the explicit scheme \eqref{eq:iso_explicit} becomes
\begin{equation}
 \bm u^{k+1} \;=\; 
 \bm u^k + \tau \bm D_h^-\left(\Phi\left(\bm D_h^+ \bm u^k\right)\right).
 \label{eq:exflux}
\end{equation}
Here, $\bm D_h^+$ and $\bm D_h^-$ represent matrices computing forward 
and backward first order derivative approximations with grid size $h$. 
Comparing \eqref{eq:exflux} to the ResNet block \eqref{eq:resnet_block} 
reveals that both perform the same computations when identifying
\begin{equation}
\bm K_1 = \bm D_h^+, \qquad 
\sigma_1 = \tau\Phi, \qquad 
\bm K_2 = \bm D_h^-, \qquad 
\bm b_1 = \bm b_2 = \bm 0, \qquad 
\sigma_2 = \text{Id}\,.
\end{equation}
The computational graph for this is shown in Figure 
\ref{fig:isotropicblock}(a).

Alt et al.~\cite{ASAPW21} use this connection to advocate ResNet 
architectures with mirrored kernels $\bm K_1$ and $\bm K_2$ to guarantee 
stability in the Euclidean norm. Moreover, their experiments show 
advantages of nonmonotone activation functions.

\medskip
\noindent
\textbf{Translating 2-D Anisotropic Diffusion into ResNets.} 
Our directional splitting allows also a natural translation of anisotropic 
diffusion into ResNets. We split the divergence term of 2-D aniso\-tropic 
diffusion into a sum of four divergence terms of 1-D diffusion 
processes, and use the previous translation for each. This is illustrated in 
Figure~\ref{fig:isotropicblock}(b). By appropriately concatenating 
the 2-D convolution kernels into 4-D tensors, we match the ResNet 
structure precisely. 

\medskip
\noindent
{\bf Experiments.}
Implementing numerical schemes for GPUs using CUDA can be labour-intensive 
and requires expertise. However, deep learning frameworks are capable of 
fully automatic and efficient parallelisation of user code. As we were able 
to decompose our discretisation into neural network primitives, we can use 
these frameworks to obtain an efficient implementation with little effort.

\changed{As a prototypical anisotropic diffusion process as described in 
\eqref{eq:aniso_continuous} we use 
edge-enhancing image diffusion (EED) \cite{We97}, for which we consider $10$ 
iterations of an explicit scheme. We compare three implementations:}
The first uses C and runs on the CPU. It computes entries of the 
$\delta$-stencil before applying it to the image. The second is an 
implementation in the PyTorch framework which follows the same strategy. As 
this style is uncommon in most neural networks, the implementation is fairly 
involved. The third also uses PyTorch, but follows our ResNet translation. It 
only requires two convolutions, one activation function, and a summation. This 
leads to a concise and simple implementation.

For an image with $2048 \times 2048$ pixels, our C code takes $1.6\,s$ on 
an AMD 5800X CPU. Both PyTorch implementations perform one order of 
magnitude faster at $0.16\,s$ and $0.15\,s$ respectively on an Nvidia 
3090 GPU. 
This demonstrates that our ResNet translation is able to significantly 
accelerate EED with a straightforward parallel implementation. It is 
even as fast as the much more involved stencil-based GPU implementation. This 
behaviour is consistent across image and batch sizes, provided that the total 
pixel count is sufficiently large.
\section{Conclusions}
We have explored three aspects of anisotropic diffusion stencils. The first 
was an intuitive derivation of a large second-order stencil family  
based on directional splitting. While it covers the full stencil class of 
Weickert et al.~\cite{WWW13}, its derivation is simpler, and it requires 
only one free parameter ($\delta$) instead of two. Therefore, we call it
the $\delta$-stencil family.

Secondly, we have established a rigorous spectral norm estimate of the
matrix associated to this stencil family. It allows to derive fairly
tight time step size limits of explicit schemes to guarantee stability 
in the Euclidean norm. We have restricted ourselves to explicit 
schemes, since they are structurally similar to feedforward neural 
networks. Moreover, they form the backbone of acceleration methods based 
on super time stepping~\cite{WGSB16} and extrapolation concepts~\cite{HOWR16}.
Further multigrid-like acceleration options may arise from multiscale
network features such as pooling operations and U-net structures 
\cite{ASAPW21}.

Thirdly, the directional splitting from our derivation has been 
instrumental in linking anisotropic diffusion to ResNets. It paves the 
road to an effortless and efficient parallelisation with libraries such 
as PyTorch. This also illustrates the usefulness of neural networks 
outside their original field of machine learning.

In our ongoing work, we also address the reverse direction: We 
investigate how neural architectures can benefit from the integration 
of anisotropic diffusion. 

\medskip
\noindent
\textbf{Acknowledgements.} We thank Kristina Schaefer for careful 
proofreading. 

\bibliographystyle{bibtex/splncs03.bst}
\bibliography{myrefs.bib}
\end{document}